%% file: main.tex
\documentclass[11pt]{amsart}

\input{includes/preamble}

\begin{document}

\begin{abstract}
We prove in ZFC that every torsion-free Abelian group of cardinality
$\cfrak$ admits a Hausdorff countably compact group topology without
nontrivial convergent sequences.  In particular, this applies to the free
Abelian group $\Z^{(\cfrak)}$, the Baer--Specker group $\Z^\omega$ and
$\Q^{(\cfrak)}$. For the topology constructed on $\Z^{(\cfrak)}$, the
coordinatewise nonnegative cone is countably compact in the subspace
topology. Consequently, there exists
in ZFC a commutative Tychonoff countably compact topological semigroup
which has two-sided cancellation but is not a group, giving a negative
answer to Wallace's question. Combined with earlier results, the main
theorem also yields in ZFC a Tychonoff countably compact topological
semigroup containing a copy of the bicyclic semigroup and a functionally
Hausdorff countably compact paratopological group that is not a
topological group.
\end{abstract}

\maketitle

\input{sections/01-introduction}
\input{sections/02-preliminaries-and-topological-reduction}

\input{sections/03-uniform-finite-approximation-by-homomorphisms}
\input{sections/04-bounded-independence-and-deletion}
\input{sections/05-triangular-coding-and-block-ultrafilters}
\input{sections/06-local-homomorphism-fusion}
\input{sections/07-transfinite-extension}
\input{sections/08-group-topologies}
\input{sections/09-wallace-semigroup}
\input{sections/10-consequences-and-concluding-remarks}
\input{sections/10-acknowledgments}

\bibliographystyle{amsplain}
\bibliography{includes/references}

\end{document}

%% file: includes/preamble.tex
\usepackage[T1]{fontenc}
\usepackage{lmodern}
\usepackage{microtype}
\usepackage{enumitem}
\usepackage{amsmath,amssymb,mathtools,mathrsfs}
\usepackage[colorlinks=true,linkcolor=blue,citecolor=blue,urlcolor=blue]{hyperref}

\setlist[enumerate,1]{label=\textup{(\roman*)},ref=\textup{(\roman*)}}

\newtheorem{theorem}{Theorem}[section]
\newtheorem{corollary}[theorem]{Corollary}
\newtheorem{lemma}[theorem]{Lemma}
\newtheorem{proposition}[theorem]{Proposition}
\newtheorem*{claim}{Claim}
\theoremstyle{definition}
\newtheorem{definition}[theorem]{Definition}

\newcommand{\T}{\mathbb T}
\newcommand{\N}{\mathbb N}
\newcommand{\Z}{\mathbb Z}
\newcommand{\Q}{\mathbb Q}
\newcommand{\cfrak}{\mathfrak c}
\newcommand{\Hom}{\operatorname{Hom}}
\newcommand{\supp}{\operatorname{supp}}
\newcommand{\Rel}{\operatorname{Rel}}
\newcommand{\htop}{\operatorname{ht}}

\title[The Wallace problem and torsion-free groups]
{The Wallace problem and countably compact torsion-free
Abelian groups in ZFC}

\author{Juliane Trianon Fraga}
\address{Independent researcher, S\~ao Paulo, SP, Brazil}
\email{julianetrianon@gmail.com}

\author{Vinicius de Oliveira Rodrigues}
\address{Department of Mathematics, Institute of Mathematics, Statistics and
Computer Science, University of S\~ao Paulo, Rua do Mat\~ao, 1010, 05508-090
S\~ao Paulo, SP, Brazil}
\email{vinior@ime.usp.br}
\date{}

\subjclass[2020]{Primary 22A05, 54H11, 54D20, 22A15; Secondary 20K20, 54A20}
\keywords{countably compactness, topological groups, torsion-free Abelian groups,nontrivial convergent sequence,
Wallace semigroup}

%% file: sections/01-introduction.tex
\section{Introduction}

A classical theorem of Gelbaum, Kalisch, and Olmsted and of Numakura states
that every compact Hausdorff topological semigroup with two-sided cancellation
is a topological group
\cite{GelbaumKalischOlmsted1951,Numakura1952}.  At the 1953 annual meeting of
the American Mathematical Society, A.~D. Wallace observed that it was not
known whether compactness could be replaced by countable compactness.  He
recorded the question in print in 1955 \cite[p.~101]{Wallace1955}; it was later
listed as Question~3L.1 by Comfort in \emph{Open Problems in Topology}
\cite{Comfort1990}:

\begin{quote}
Must every Hausdorff countably compact topological semigroup with two-sided
cancellation be a topological group?
\end{quote}

A negative example is usually called a \emph{Wallace semigroup}.  The problem
is delicate because several natural strengthenings of countable compactness do
force the compact conclusion.  For example, the conclusion holds under first
countability and under sequential compactness
\cite{MukherjeaTserpes1972,Grant1993}.  On the other hand,
by assuming additional set-theoretic hypotheses, several consistent
constructions of Wallace semigroups have been obtained.  Robbie and
Svetlichny constructed a Wallace semigroup under the Continuum Hypothesis, and
Tomita obtained one under Martin's Axiom for countable partial orders
\cite{RobbieSvetlichny1996,Tomita1996Wallace}.  However, the question of
whether a Wallace semigroup exists in the standard Zermelo--Fraenkel set
theory with the Axiom of Choice (ZFC) remained open for 73 years.

The route from topological groups to Wallace semigroups passes through a
second long-standing problem.  A countably compact Boolean group without
nontrivial convergent sequences was constructed in ZFC by Hru\v{s}\'ak, van Mill,
Ramos-Garc\'{\i}a, and Shelah \cite{HrusakVanMillRamosShelah2021}.
That construction resolved the general existence problem posed by van Douwen
and, together with the reductions of van Douwen and Tomita, the product problem
of Comfort \cite{vanDouwen1980,Tomita2005Square}.
However, it did not yield a Wallace semigroup.  Accordingly, Hru\v{s}\'ak,
van Mill, Ramos-Garc\'{\i}a, and Shelah asked whether
there exists in ZFC a non-torsion countably compact topological group without
nontrivial convergent sequences
\cite[Questions~5.5 and~5.6]{HrusakVanMillRamosShelah2021}; see also
\cite[Questions~19 and~20]{HrusakShibakov2021}.

Free Abelian groups form the canonical torsion-free test case.  In 1990,
Comfort asked whether the free Abelian group on $\cfrak$ generators admits a
countably compact group topology in ZFC, attributing the question to Tkachenko
\cite[Question~3C.1]{Comfort1990}.  Tkachenko had constructed such a topology
under the Continuum Hypothesis \cite{Tkachenko1990}.  Later constructions used
Martin's Axiom or selective ultrafilters
\cite{KoszmiderTomitaWatson2000,MadariagaGarciaTomita2007,
BoeroPereiraTomita2019}.  Stronger examples with prescribed compactness of
finite powers were also obtained from additional set-theoretic hypotheses
\cite{BelliniHartRodriguesTomita2023,FuentesMaguinaRodriguesTomita2024}.
The ZFC question for the free Abelian group, including the stronger requirement
of having no nontrivial convergent sequences, remained open
\cite[Problem~6.2]{BelliniRodriguesTomita2021Forcing} and
\cite[Question~6.2]{BelliniHartRodriguesTomita2023}.

Our result treats the entire class of
torsion-free Abelian groups of cardinality continuum.

\begin{theorem}\label{thm:main}
In ZFC every torsion-free Abelian group of cardinality $\cfrak$ admits a
Hausdorff countably compact group topology in which every convergent sequence
is eventually constant.
\end{theorem}

Theorem~\ref{thm:main} constructs in
ZFC a torsion-free Abelian countably compact group with no
non-trivial convergent sequences. Before this result, the known
ZFC constructions produced only torsion groups. It was therefore
unknown whether there exists in ZFC a non-torsion countably
compact group with no non-trivial convergent sequences. This question was
posed in \cite[Question~6(a)]{Tomita2019SmallPowers},
\cite[Question~20]{HrusakShibakov2021},
\cite[Question~5.6]{HrusakVanMillRamosShelah2021}, and
\cite[Question~6.8]{TomitaTrianonFraga2022}.
Theorem~\ref{thm:main} answers it affirmatively. Moreover, since the groups
constructed here are torsion-free, the theorem also proves the existence in
ZFC of a torsion-free countably compact group with no non-trivial
convergent sequences, as asked in
\cite[Problem~1.2]{BelliniRodriguesTomita2021Forcing}.

Several classical groups illustrate the gain in generality.  Taking
$G=\Z^{(\cfrak)}$ resolves Tkachenko's free Abelian problem in ZFC, with the
additional absence of nontrivial convergent sequences.  Taking
$G=\Z^\omega$ answers Question~14.7(i) of Dikranjan and Shakhmatov about the
Baer--Specker group \cite{DikranjanShakhmatov2005}.
We also have the following consequence for the additive group of real numbers.

\begin{proposition}\label{prop:rational}
In ZFC the additive group $\Q^{(\cfrak)}\approx\mathbb R$ admit Hausdorff
countably compact group topologies in which every convergent sequence is
eventually constant.
\end{proposition}

We apply the construction to
$\Z^{(\cfrak)}$ and prove that its nonnegative cone is countably compact in the
induced topology.  Thus the counterexample is commutative and Tychonoff in
addition to satisfying the requirements of the problem.

\begin{corollary}\label{cor:wallace}
In ZFC there exists a commutative Tychonoff countably compact topological
semigroup with two-sided cancellation which is not a group.
\end{corollary}

This gives a negative answer in ZFC to Wallace's question
\cite{Wallace1955,Comfort1990}.  It also settles its later ZFC formulations in
\cite[p.~12]{Shakhmatov2001}, \cite[p.~524]{Tkachenko2002Between},
\cite[Problem~523]{Pearl2004},
\cite[Question~5.1(a)]{BoeroPereiraTomita2019},
\cite[Question~19]{HrusakShibakov2021},
\cite[Question~5.5]{HrusakVanMillRamosShelah2021},
\cite[Problem~6.4]{BelliniRodriguesTomita2021Forcing},
\cite[Question~6.7]{TomitaTrianonFraga2022}, and
\cite[Question~6.3]{BelliniHartRodriguesTomita2023}.

The same theorem has further consequences which are not specializations to a
single familiar group.  Known results reduce problems about the
bicyclic semigroup, countably compact paratopological groups, and monothetic
monoids to the existence of a torsion-free Abelian countably compact group
without nontrivial convergent sequences.  Section~\ref{sec:consequences}
records these consequences and a consequence for suitable sets.  There we also
answer a question of Shakhmatov: whether there exists in ZFC a countably
compact free Abelian group in which every infinite closed subset has
cardinality at least $\cfrak$ \cite[p.~13]{Shakhmatov2001}.  We prove that
the topological group constructed on $\Z^{(\cfrak)}$ has this property.
We also delimit questions for which Theorem~\ref{thm:main} gives a partial answer.

%% file: sections/02-preliminaries-and-topological-reduction.tex
\section{Preliminaries and the topological reduction}\label{sec:preliminaries}

We write $\N=\{0,1,2,\ldots\}$, regard each $m\in\N$ as the finite ordinal
$\{i:i<m\}$, and index every $m$-tuple by $i<m$.  We use additive notation.
The circle group
$\T=\mathbb R/\mathbb Z$ carries the quotient metric induced by the Euclidean
metric on $\mathbb R$, determined by
\[
 \|x+\mathbb Z\|=\min_{n\in\mathbb Z}|x-n|
 \qquad(x\in\mathbb R).
\]
If $(x_n:n\in\N)$ is a sequence in a Hausdorff topological space and $p$ is
an ultrafilter on $\N$, then
$p\text{-}\lim_n x_n=x$ means that
\[
 \text{for every neighborhood }U\text{ of }x,\ \{n:x_n\in U\}\in p.
\]
Every ultrafilter on a compact Hausdorff space converges to a unique point.
An ultrafilter on $\N$ is free when it
contains the cofinite filter.

The circle is divisible and therefore injective in the category of Abelian
groups.  Consequently a homomorphism from a subgroup of an Abelian group to
$\T$ extends to the whole group
\cite[Theorem~21.1]{Fuchs1970}.
We write $\Hom(G,\T)$ for the set of group homomorphisms from an Abelian
group $G$ to $\T$.

As usual, $\Q^{(\cfrak)}=\{z\in\Q^{\cfrak}:|\supp(z)|<\infty\}$ denotes the
direct sum of $\cfrak$ copies of $\Q$, where
$\supp(z)=\{\xi<\cfrak:z(\xi)\neq0\}$ is the support of $z$.

We adopt the convention that, if $D\subseteq\cfrak$, then $\Q^{(D)}$
denotes the subgroup of $\Q^{(\cfrak)}$ consisting of all vectors supported
in $D$.  Thus, each element of $\Q^{(D)}$ is a function from $\cfrak$ into
$\Q$ that vanishes outside $D$.

For each $\xi<\cfrak$, let $e_\xi\in\Z^{(\cfrak)}$ denote the vector
supported at $\xi$ and defined by
\[
 e_\xi(\eta)=
 \begin{cases}
  1,&\text{if $\eta=\xi$},\\
  0,&\text{if $\eta\neq\xi$}.
 \end{cases}
\]

\begin{lemma}\label{lem:torsion-free-coordinatization}
Let $G$ be a torsion-free Abelian group of cardinality $\cfrak$.  Then $G$ may
be identified with a subgroup of $\Q^{(\cfrak)}$ in such a way that
\[
 \Z^{(\cfrak)}\leq G\leq\Q^{(\cfrak)}.
\]
In particular, for every $\xi<\cfrak$, the vector $e_\xi$ belongs to $G$.
\end{lemma}

\begin{proof}
The canonical homomorphism
\[
 G\longrightarrow\Q\otimes_{\Z}G,
 \qquad x\longmapsto1\otimes x,
\]
is injective because $G$ is torsion-free.  Its image spans the rational vector
space $\Q\otimes_{\Z}G$. As $\mathbb Q$ is countable, the dimension of this vector space is $\cfrak$ because $G$ has cardinality $\cfrak$.

Choose a basis $(1\otimes b_\xi)_{\xi<\cfrak}$ from the spanning set given by
the image of $G$, and identify $\Q\otimes_{\Z}G$ with
$\Q^{(\cfrak)}$ by sending $1\otimes b_\xi$ to $e_\xi$.  Under the resulting
embedding of $G$, each $b_\xi$ is identified with $e_\xi$, and hence
$\Z^{(\cfrak)}\leq G\leq\Q^{(\cfrak)}$.
\end{proof}

For a finite tuple $z=(z_0,\ldots,z_{m-1})$ in an Abelian group put
\[
 \Rel(z)=\left\{a\in\Z^m:\sum_{i<m} a_i z_i=0\right\},
 \qquad
 \htop(a)=\max\bigl(\{0\}\cup\{|a_i|:i<m\}\bigr).
\]

In a $T_1$ space, countable
compactness is equivalent to the assertion that every countably infinite
subset has an accumulation point
\cite[Theorem~3.10.3(v)]{Engelking1989}.  Also, if $u:\N\to X$ is injective,
$p$ is free, and $p\text{-}\lim_n u(n)=x$, then $x$ is an accumulation point
of $u[\N]$.

The following observation will be applied to show that the group topologies
we construct are countably compact.
This is a standard result and we include a proof for completeness.

\begin{lemma}\label{lem:topological-reduction}
Let $G$ be a Hausdorff topological group.  Suppose that for every injective
$s:\N\to G$ there exist a strictly increasing map $\varphi:\N\to\N$, a free
ultrafilter $p$ on $\N$, and a point $b\neq0$ such that
\begin{equation}\label{eq:nonzero-limit-property}
 p\text{-}\lim_n s(\varphi(n))=b.
\end{equation}
Then $G$ is countably compact and every convergent sequence in $G$ is
eventually constant.
\end{lemma}

\begin{proof}
Given a countably infinite set $A\subseteq G$, enumerate it injectively by
$s$ and apply the hypothesis.  It follows that $A$ has an accumulation point,
so $G$ is countably compact.

Now suppose that an injective sequence $s$ converges to $x$.  Define
$t:\N\to G$ by $t(n)=s(n)-x$.
The sequence $t$ is injective and converges to $0$.  Apply the hypothesis
to $t$, obtaining $\varphi,p$ and $b\neq0$.  Since $\varphi$ is strictly
increasing, $t\circ\varphi$ still converges to $0$, and since $p$ is
free, it follows that $p\text{-}\lim_n t(\varphi(n))=0$.  This contradicts
uniqueness of limits in a Hausdorff space, so $s$ cannot be injective.

Finally let $s$ be any sequence converging to $x$.  If its range is finite,
then
$s(n)=x$ eventually.  If its range is infinite, choose recursively a
strictly increasing subsequence with pairwise distinct values.  That
subsequence still converges to $x$, contradicting the preceding paragraph.
\end{proof}

Suppose now that $(\psi_j)_{j\in J}$ is a family of homomorphisms in
$\Hom(G,\T)$ that separates points: whenever $x\neq y$, there exists $j\in J$
such that $\psi_j(x)\neq\psi_j(y)$.  Define
\[
 \Delta:G\longrightarrow\T^J,
 \qquad \Delta(x)=(\psi_j(x))_{j\in J}.
\]
Give $G$ the coarsest topology that makes every $\psi_j$ continuous, that is,
the initial topology induced by the family $(\psi_j)_{j\in J}$. This is a
Hausdorff group topology, and $\Delta$ is a topological embedding
onto its image.  Thus, once these homomorphisms have been constructed, only
\eqref{eq:nonzero-limit-property} remains to obtain all the topological
conclusions.

%% file: sections/03-uniform-finite-approximation-by-homomorphisms.tex
\section{Uniform finite approximation by circle-valued homomorphisms}

A basic ingredient in the construction is the
following finite interpolation problem.  Given elements
$z_0,\ldots,z_{m-1}$ of an Abelian group and prescribed target values
$t_0,\ldots,t_{m-1}\in\T$, we seek a single homomorphism from the ambient
group to $\T$ whose value at each $z_i$ is ``close'' to $t_i$.

The following Kronecker-type lemma shows
that, for approximate realization to a fixed accuracy, it is enough to check
compatibility only with the relations whose coefficients are bounded by a
suitable finite constant.

\begin{lemma}[Uniform finite Kronecker lemma]\label{lem:uniform-kronecker}
For every integer $m\geq1$ and every $\varepsilon>0$ there exists $q\in\N$
with the following property.

For every Abelian group $G$, for every
$z=(z_i)_{i<m}\in G^m$, and for every
$t=(t_i)_{i<m}\in\T^m$, if for every relation
$a=(a_i)_{i<m}\in\Rel(z)$ with $\htop(a)\leq q$ we have
$\sum_{i<m}a_i t_i=0$, then there exists $\psi\in\Hom(G,\T)$ such that
\[
 \max_{i<m}\|\psi(z_i)-t_i\|<\varepsilon.
\]
\end{lemma}
We note that the integer $q$ is independent of $G,z$, and $t$.

We dedicate this section to proving Lemma~\ref{lem:uniform-kronecker}.
First, we isolate the analytic ingredients that will be needed.  For each
integer $m\geq1$, on $\T^m$ we use the maximum metric
\[
 \rho_\infty(u,v)=\max_{i<m}\|u_i-v_i\|.
\]
Let $\mathbb S^1=\{w\in\mathbb C:|w|=1\}$.  The homomorphism
$r\mapsto e^{2\pi i r}$ from $\mathbb R$ onto $\mathbb S^1$ has kernel
$\mathbb Z$, so it induces the continuous group isomorphism
\[
 E:\T=\mathbb R/\mathbb Z\longrightarrow\mathbb S^1,
 \qquad E(r+\mathbb Z)=e^{2\pi i r}.
\]
For $a=(a_i)_{i<m}\in\Z^m$, define the continuous homomorphism
\[
 \phi_a:\T^m\longrightarrow\mathbb S^1,
 \qquad \phi_a(u)=E\left(\sum_{i<m}a_i u_i\right).
\]

The only compactness result for families of functions that we need is the
following compact-metric form of Arzel\`a--Ascoli.  For a proof, see, e.g.,
\cite[Theorem~8.2.10]{Engelking1989}.

\begin{lemma}[Arzel\`a--Ascoli]\label{lem:arzela-ascoli-uniform}
Let $X$ be a compact metric space.  If
$\mathcal F\subseteq C(X,\mathbb R)$ is equicontinuous and uniformly bounded,
then its closure in the uniform norm is compact.
\end{lemma}

For the following standard complex form of the Stone--Weierstrass theorem,
see \cite[Theorem~7.33]{Rudin1976}.  Recall that a subset
$\mathcal A\subseteq C(X,\mathbb C)$ is uniformly dense in
$C(X,\mathbb C)$ if it is dense in the topology induced by the uniform norm
$\|\cdot\|_\infty$.

\begin{theorem}[Complex Stone--Weierstrass theorem]
\label{thm:complex-stone-weierstrass}
Let $X$ be a compact Hausdorff space, and let
$\mathcal A\subseteq C(X,\mathbb C)$ be a complex subalgebra.  If
$\mathcal A$ contains the constant functions, separates the points of $X$,
and is closed under complex conjugation, then $\mathcal A$ is uniformly dense
in $C(X,\mathbb C)$.
\end{theorem}

\begin{lemma}
\label{lem:trigonometric-stone-weierstrass}
For every integer $m\geq1$, the complex linear span of
$\{\phi_a:a\in\Z^m\}$ is uniformly dense in $C(\T^m,\mathbb C)$.
\end{lemma}

\begin{proof}
Their complex linear span is a unital algebra because
$\phi_0=1$ and $\phi_a\phi_b=\phi_{a+b}$, and it is closed under complex
conjugation because $\overline{\phi_a}=\phi_{-a}$.  It also separates
points: if $u\neq v$, then $u_j\neq v_j$ for some $j<m$, and the homomorphism
$\phi_{e_j}$, where $e_j$ is the $j$th standard basis vector, distinguishes
$u$ and $v$.
Theorem~\ref{thm:complex-stone-weierstrass} now applies.
\end{proof}

We will call each element of the span of $\{\phi_a:a\in\Z^m\}$ a
\emph{trigonometric polynomial} on $\T^m$.

We now combine these results to obtain the lemma below.
In its statement, $1$-Lipschitz means Lipschitz with constant $1$ with respect
to the metric $\rho_\infty$ on $\T^m$ and the usual metric on $\mathbb R$.

\begin{lemma}\label{lem:uniform-frequency-set}
For every integer $m\geq1$ and every $\eta>0$, there exists a finite set
$S\subseteq\Z^m$ such that for every $1$-Lipschitz function
$f:\T^m\to[0,1/2]$ there exist complex coefficients $(c_a)_{a\in S}$
satisfying
\[
 \left\|f-\sum_{a\in S}c_a\phi_a\right\|_\infty<\eta.
\]
\end{lemma}

\begin{proof}
Let $\mathcal L_m$ be the set of all such functions, regarded as a subset of
$C(\T^m,\mathbb R)$.  This family is uniformly bounded, equicontinuous (as
all its members are $1$-Lipschitz), and closed under uniform limits.  Hence it
is closed in $C(\T^m,\mathbb R)$ and compact by
Lemma~\ref{lem:arzela-ascoli-uniform}.

For each $f\in\mathcal L_m$, use
Lemma~\ref{lem:trigonometric-stone-weierstrass} to choose a trigonometric
polynomial $P_f$ such that $\|f-P_f\|_\infty<\eta/2$.  The uniform balls of
radius $\eta/2$ centered at the members of $\mathcal L_m$ cover
$\mathcal L_m$; choose a finite subcover with centers
$f_0,\ldots,f_N$.
For each $j\leq N$, choose a finite set $S_j\subseteq\Z^m$ such that
$P_{f_j}\in\operatorname{span}\{\phi_a:a\in S_j\}$, and set
$S=\bigcup_{j\leq N} S_j$.

If $f$ belongs to the ball centered at $f_j$, then
\[
 \|f-P_{f_j}\|_\infty
 \leq \|f-f_j\|_\infty+\|f_j-P_{f_j}\|_\infty<\eta.
\]
This proves
the assertion.
\end{proof}

For an integer $m\geq1$ and a subgroup $R\leq\Z^m$, put
\[
 H_R=\left\{u\in\T^m:\sum_{i<m}a_i u_i=0
                    \text{ for every }a\in R\right\}.
\]

Recall that if $H$ is a compact Hausdorff Abelian topological group, its
\emph{normalized Haar measure} is the unique regular Borel measure $\mu$
on $H$ such that $\mu(H)=1$ and
$\mu(x+B)=\mu(B)$ for every $x\in H$ and every Borel set $B\subseteq H$.
Thus $\mu$ is a probability measure invariant under translations; in
particular, for every continuous $g:H\to\mathbb C$ and every $x\in H$,
\[
 \int_H g(x+h)\,d\mu(h)=\int_H g(h)\,d\mu(h).
\]
Existence and uniqueness are standard; see
\cite[Chapter~VII, \S\S~2--3]{Katznelson2004}.

\begin{lemma}
\label{lem:annihilator-haar-average}
For every integer $m\geq1$ and every subgroup $R\leq\Z^m$, the set $H_R$ is
a compact subgroup of $\T^m$, and
\begin{equation}\label{eq:double-annihilator}
 \{a\in\Z^m:\phi_a(h)=1\text{ for every }h\in H_R\}=R.
\end{equation}
If $\mu_R$ is normalized Haar measure on $H_R$, then, for every
$a\in\Z^m$ and $s\in\T^m$,
\begin{equation}\label{eq:haar-homomorphism-average}
 \int_{H_R}\phi_a(s+h)\,d\mu_R(h)
 =\begin{cases}
   \phi_a(s),&a\in R,\\
   0,&a\notin R.
  \end{cases}
\end{equation}
\end{lemma}

\begin{proof}
The set $H_R$ is an intersection of kernels of continuous homomorphisms
from $\T^m$ to $\T$, so it is a closed subgroup of the compact group
$\T^m$.  The inclusion from right to left in
\eqref{eq:double-annihilator} follows from the definition of $H_R$.

For the reverse inclusion, let $\pi:\Z^m\to\Z^m/R$ be the quotient map and
suppose that $a\notin R$.  Since $\pi(a)\neq0$, there is a homomorphism
$\phi:\Z^m/R\to\T$ such that $\phi(\pi(a))\neq0$: first define it on
$\langle\pi(a)\rangle$ and then use the injectivity of $\T$ to extend it.
For $i<m$, put $u_i=\phi(\pi(e_i))$.  If $r=(r_i)_{i<m}\in R$, then
\[
 \sum_{i<m}r_i u_i
 =\phi\left(\pi\left(\sum_{i<m}r_i e_i\right)\right)
 =\phi(\pi(r))=0,
\]
so $u=(u_i)_{i<m}\in H_R$.  On the other hand,
\[
 \sum_{i<m}a_i u_i=\phi(\pi(a))\neq0.
\]
Thus $\phi_a(u)\neq1$, proving \eqref{eq:double-annihilator}.

Put
\[
 I_a=\int_{H_R}\phi_a(h)\,d\mu_R(h).
\]
If $a\in R$, then $\phi_a$ is identically $1$ on $H_R$, so $I_a=1$.  If
$a\notin R$, identity \eqref{eq:double-annihilator} gives
$h_0\in H_R$ with $\phi_a(h_0)\neq1$.  Translation invariance gives
\[
 I_a=\int_{H_R}\phi_a(h_0+h)\,d\mu_R(h)=\int_{H_R}\phi_a(h_0)\phi_a(h)\,d\mu_R(h)=\phi_a(h_0)I_a,
\]
and hence $I_a=0$.  Finally,
$\phi_a(s+h)=\phi_a(s)\phi_a(h)$, so
\[
  \int_{H_R}\phi_a(s+h)\,d\mu_R(h)=\phi_a(s)I_a,
\]
which proves \eqref{eq:haar-homomorphism-average}.
\end{proof}

Now we are ready to prove Lemma~\ref{lem:uniform-kronecker}.

\begin{proof}[Proof of Lemma~\ref{lem:uniform-kronecker}]
Set $\eta=\varepsilon/4$, and let $S\subseteq\Z^m$ be supplied by
Lemma~\ref{lem:uniform-frequency-set}.
Put
\[
 q=\max\bigl(\{1\}\cup\{\htop(a):a\in S\}\bigr).
\]

\begin{claim}
For every subgroup $R\leq\Z^m$ and every
$t\in\T^m$, if
\[
 \sum_{i<m}a_i t_i=0
 \qquad\text{for every }a\in R\cap[-q,q]^m,
\]
then $\rho_\infty(t,H_R)<\varepsilon$.
\end{claim}

\begin{proof}[Proof of the claim]
Fix a subgroup $R\leq\Z^m$ and $t\in\T^m$ such that
\begin{equation}\label{eq:short-annihilation}
 \sum_{i<m}a_i t_i=0\text{ for every }a\in R\cap[-q,q]^m.
\end{equation}
Define $d:\T^m\to\mathbb R$ by
\[
 d(u)=\rho_\infty(u,H_R)=\inf_{h\in H_R}\rho_\infty(u,h).
\]
The function $d$ is $1$-Lipschitz.
Moreover, $0\in H_R$, and hence
\[
 0\leq d(u)\leq\rho_\infty(u,0)\leq\frac12
 \qquad(u\in\T^m).
\]
Lemma~\ref{lem:uniform-frequency-set}, applied to $d$, gives coefficients
$(c_a)_{a\in S}\in\mathbb C^S$ such that the trigonometric polynomial
\[
 P(u):=\sum_{a\in S}c_a\phi_a(u)
\]
satisfies $\|P-d\|_\infty<\eta$.

For $s\in\T^m$, set
\[
 A(s)=\int_{H_R}P(s+h)\,d\mu_R(h).
\]
Equation~\eqref{eq:haar-homomorphism-average} yields
\begin{equation}\label{eq:average-polynomial}
 A(s)=\sum_{a\in S\cap R}c_{a}\phi_a(s).
\end{equation}
By the definition of $q$,
$S\cap R\subseteq R\cap[-q,q]^m$.  Thus
\eqref{eq:short-annihilation} implies $\phi_a(t)=1$ for every
$a\in S\cap R$, and \eqref{eq:average-polynomial} gives
\begin{equation}\label{eq:equal-averages}
 A(t)=\sum_{a\in S\cap R}c_a\phi_a(t)
     =\sum_{a\in S\cap R}c_a
     =\sum_{a\in S\cap R}c_a\phi_a(0)=A(0).
\end{equation}

Because $d$ vanishes on $H_R$ and $\mu_R$ is a probability measure,
\begin{equation}\label{eq:average-at-zero}
 |A(0)|
 =\left|\int_{H_R}\bigl(P(h)-d(h)\bigr)\,d\mu_R(h)\right|\leq\|P-d\|_\infty<\eta.
\end{equation}
Since $H_R$ is a subgroup and $\rho_\infty$ is translation invariant,
$d(t+h)=d(t)$ for every $h\in H_R$.  Consequently,
$d(t)=\int_{H_R}d(t+h)\,d\mu_R(h)$, and therefore
\begin{equation}\label{eq:average-at-t}
 |d(t)-A(t)|
 =\left|\int_{H_R}\bigl(d(t+h)-P(t+h)\bigr)\,d\mu_R(h)\right|
 <\eta.
\end{equation}
Combining \eqref{eq:equal-averages}--\eqref{eq:average-at-t}, we obtain
\[
 d(t)=|d(t)|
 \leq |d(t)-A(t)|+|A(0)|
 <2\eta<\varepsilon.
\]
\end{proof}

Now let $G$ be an Abelian group, let
$z=(z_i)_{i<m}\in G^m$ and
$t=(t_i)_{i<m}\in\T^m$ satisfy the hypothesis of the lemma, and take
$R=\Rel(z)$.
Then $t$ satisfies the hypothesis of the claim for this $R$.  Hence $\rho_\infty(t,H_R)<\varepsilon$, so there
exists $u\in H_R$ such that
$\rho_\infty(u,t)<\varepsilon$.  Define a map on the subgroup
$L=\langle z_i:i<m\rangle\leq G$ by
\[
 \varphi\left(\sum_{i<m}n_i z_i\right)=\sum_{i<m}n_i u_i.
\]
To see that this is well defined, suppose that $n,n'\in\Z^m$ satisfy
$\sum_{i<m}n_i z_i=\sum_{i<m}n'_i z_i$.  Then $n-n'\in\Rel(z)=R$, and since
$u\in H_R$,
\[
 \sum_{i<m}n_i u_i-\sum_{i<m}n'_i u_i
 =\sum_{i<m}(n_i-n'_i)u_i=0.
\]
Clearly, $\varphi$ is a homomorphism from $L$ to $\T$.
By injectivity of $\T$, the map $\varphi$ extends to some
$\psi\in\Hom(G,\T)$.  For each $i<m$ we have $\psi(z_i)=u_i$, and therefore
\[
 \max_{i<m}\|\psi(z_i)-t_i\|=\rho_\infty(u,t)<\varepsilon.
\]
\end{proof}

%% file: sections/04-bounded-independence-and-deletion.tex
\section{Bounded integer relations and deletion}

The goal of this section is Lemma~\ref{lem:fusion-step}.  Given a
homomorphism $\psi:G\to\T$ and finite sets $A,X\subseteq G$, that lemma
produces a large subset $Y\subseteq X$ and a homomorphism $\psi':G\to\T$
which approximately preserves the values of $\psi$ on $A$ and is small on
$Y$.  Lemma~\ref{lem:uniform-kronecker} reduces this construction to
eliminating bounded-height integer relations that mix elements of $A$ with
elements of $Y$.  We now develop the finite combinatorial deletion argument
needed to achieve this.

\begin{definition}
Let $G$ be an Abelian group.  For a finite set $X\subseteq G$, put
\[
 \Rel_G(X)=
 \left\{c=(c_x)_{x\in X}\in\Z^X:
 \sum_{x\in X}c_x\,x=0\right\}.
\]
This is the set-indexed version of the notation $\Rel(z)$ introduced in
Section~\ref{sec:preliminaries}.
For $c\in\Z^X$, put
\[
 \htop(c)=\max\bigl(\{0\}\cup\{|c_x|:x\in X\}\bigr).
\]

For $M\in\N$, the set $X$ is \emph{$M$-independent} if the zero vector is
the only $c\in\Rel_G(X)$ satisfying $\htop(c)\leq M$.

For finite sets $A,Y\subseteq G$, put
\[
 \Rel_G(A,Y)=
 \left\{(b,c)\in\Z^A\times\Z^Y:
 \sum_{a\in A}b_a\,a+\sum_{y\in Y}c_y\,y=0\right\}.
\]
An element $(b,c)\in\Rel_G(A,Y)$ is \emph{mixed} if $b\neq0$ and
$c\neq0$.  We set $\htop(b,c)=\max\{\htop(b),\htop(c)\}$.
\end{definition}

\begin{lemma}\label{lem:finite-extraction}
Let $G$ be a torsion-free Abelian group and $S\subseteq G$ be infinite.
Given $M,N\in\N$,
there exists an $M$-independent subset of $S$ of cardinality $N$.
\end{lemma}

\begin{proof}
We construct, by induction on $k\leq N$, an $M$-independent set
$B_k\subseteq S$ such that $|B_k|=k$.

For $k=0$, take $B_0=\varnothing$, which is trivially
$M$-independent.  Suppose that $k<N$ and that
\[
B_k=\{x_i:i<k\}\subseteq S
\]
has already been chosen and is $M$-independent. We shall choose
$x\in S\setminus B_k$ such that $B_k\cup\{x\}$ remains
$M$-independent.

Indeed, given $x\in S\setminus B_k$, if $B_k\cup\{x\}$ is not
$M$-independent, then there exist integers $q$ and $(c_i)_{i<k}$, not all
zero, with
\[
|q|\leq M,\qquad |c_i|\leq M\quad(i<k),
\]
such that
\[
qx+\sum_{i<k} c_i x_i=0.
\]
We claim that $q\neq0$. Indeed, if $q=0$, then the family
$c\in\Z^{B_k}$ defined by $c_{x_i}=c_i$ for $i<k$ would be a nonzero
element of $\Rel_G(B_k)$ with $\htop(c)\leq M$, contradicting the
$M$-independence of $B_k$. Hence every bad
choice of $x$ satisfies an equation of the form
\[
qx=-\sum_{i<k} c_i x_i,
\qquad
0<|q|\leq M,\quad |c_i|\leq M\quad(i<k).
\]

There exist only finitely many possible choices of $q$ and $(c_i)_{i<k}$,
and each such choice determines at most one possible value of $x$.
In fact, if both $x$ and $x'$ satisfy the same equation, then
\[
q(x-x')=0.
\]
Since $G$ is torsion-free and $q\neq0$, it follows that $x=x'$.
Consequently, only finitely many elements of $S$ are forbidden.

Because $S$ is infinite and $B_k$ is finite, we may choose
$x_k\in S\setminus B_k$ outside this finite set of forbidden elements.
Then $B_{k+1}=B_k\cup\{x_k\}$ is $M$-independent and has cardinality
$k+1$.  This completes the induction step and the proof.

\end{proof}

\begin{lemma}\label{lem:integer-dependence}
For $R,Q\in\N$ there is $B=B(R,Q)\in\N$ such that, whenever
$s\in\N$, $s\leq R$, and
$\mathbf v=(v_0,\ldots,v_s)\in(\Z^s)^{s+1}$ has all coordinates of
absolute value at most $Q$, there exists
$\lambda\in\Rel(\mathbf v)\setminus\{0\}$ such that
$\htop(\lambda)\leq B$.
\end{lemma}
\begin{proof}
Fix $s\leq R$ and let $\mathbf v=(v_0,\ldots,v_s)$ be a tuple satisfying
the hypotheses.  Consider the linear map
$T_{\mathbf v}:\Q^{s+1}\longrightarrow\Q^s$
defined by
\[
T_{\mathbf v}(c_0,\ldots,c_s)
   =\sum_{i\le s}c_i v_i.
\]
As $T_{\mathbf v}$ cannot be injective, there
is a nonzero tuple $(c_0,\ldots,c_s)\in\Q^{s+1}$
such that
\[
\sum_{i\le s}c_i v_i=0.
\]
Choose a positive integer $d$ which is a common multiple of the
denominators of $c_0,\ldots,c_s$, and put
\[
\lambda_i=dc_i
\qquad(i\le s).
\]
Then each $\lambda_i$ is an integer, the tuple
$(\lambda_0,\ldots,\lambda_s)$ is nonzero, and
\[
\sum_{i\le s}\lambda_i v_i
   =d\sum_{i\le s}c_i v_i
   =0.
\]
Thus $\Rel(\mathbf v)\setminus\{0\}$ is nonempty for every tuple
$\mathbf v$ satisfying the hypotheses.  Let
\[
b(\mathbf v)=
\min\left\{
 \htop(\lambda):
 \lambda\in\Rel(\mathbf v)\setminus\{0\}
\right\}
\]
and
\[
\mathcal V_s=
\left\{
 (v_0,\ldots,v_s)\in(\Z^s)^{s+1}:
 |v_i(j)|\le Q
 \text{ for every }i\le s\text{ and }j<s
\right\}.
\]
Both $\mathcal V_s$ and the set
\[
\mathcal V=
\bigcup_{s\le R}\bigl(\{s\}\times\mathcal V_s\bigr)
\]
are finite.  We may therefore define
\[
B(R,Q)=
\max\left(
 \{1\}\cup
 \left\{
 b(\mathbf v):
 s\le R\text{ and }\mathbf v\in\mathcal V_s
 \right\}
\right),
\]
which is a well-defined positive integer depending only on $R$ and
$Q$.
By the definition of
$b(\mathbf v)$, there is a nonzero
$\lambda=(\lambda_0,\ldots,\lambda_s)\in
\Rel(\mathbf v)$ such that
\[
\htop(\lambda)=b(\mathbf v)\le B(R,Q).
\]
This is the required tuple.
\end{proof}

\begin{lemma}[Bounded deletion]\label{lem:bounded-deletion}
For every $R,Q\in\N$ there exists $M=M(R,Q)\in\N$ with the following
property.  If $G$ is a torsion-free Abelian group, $A,X\subseteq G$ are
finite, $|A|\leq R$, and $X$ is
$M$-independent, then some $Y\subseteq X$ satisfies
\begin{enumerate}
\item $|X\setminus Y|\leq|A|$;
\item there is no mixed $(b,c)\in\Rel_G(A,Y)$ such that
$\htop(b,c)\leq Q$.
\end{enumerate}
\end{lemma}

\begin{proof}
Let $B=B(R,Q)$ be supplied by
Lemma~\ref{lem:integer-dependence}, and put
\[
M=(R+1)BQ.
\]
Choose a subset $Y\subseteq X$ of maximum cardinality which admits no
mixed $(b,c)\in\Rel_G(A,Y)$ with $\htop(b,c)\le Q$.

Write $s=|A|$ and suppose, towards a contradiction, that
$X\setminus Y$ contains distinct elements $x_0,\ldots,x_s$. By the
maximality of
$|Y|$, for each $i\leq s$ there is a mixed element
\[
 (\beta_i,\gamma_i)\in\Rel_G(A,Y\cup\{x_i\})
 \qquad\text{with}\qquad
 \htop(\beta_i,\gamma_i)\leq Q.
\]
Thus
\begin{equation}\label{eq:deletion-witness}
 \sum_{a\in A}\beta_i(a)\,a
 +
 \sum_{y\in Y}\gamma_i(y)\,y
 +
 \gamma_i(x_i)\,x_i
 =0.
\end{equation}
Moreover, $\gamma_i(x_i)\ne0$, since otherwise
$(\beta_i,\gamma_i|_Y)$ would be a mixed element of $\Rel_G(A,Y)$
contradicting the choice of $Y$.

Fix an arbitrary enumeration
\[
A=\{a_0,\ldots,a_{s-1}\}
\]
and for each $i\le s$, define
\[
b_i=(\beta_i(a_0),\ldots,\beta_i(a_{s-1}))\in\Z^s.
\] 

Apply Lemma~\ref{lem:integer-dependence} to
$b_0,\ldots,b_s\in\Z^s$.  Since $s\le R$, there are integers
$\lambda_0,\ldots,\lambda_s$, not all zero, such that
\[
|\lambda_i|\le B\quad(i\le s),
\qquad
\sum_{i\le s}\lambda_i b_i=0.
\]

Multiplying the $i$-th instance of \eqref{eq:deletion-witness} by $\lambda_i$
and summing over $i\leq s$, the sum involving the elements of $A$ vanishes,
and we obtain
\[
\sum_{y\in Y}
 \left(\sum_{i\le s}\lambda_i\gamma_i(y)\right)\,y
+
\sum_{i\le s}\lambda_i\gamma_i(x_i)\,x_i
=0.
\]
Put $W=Y\cup\{x_0,\ldots,x_s\}$ and define $e\in\Z^W$ by
\[
 e_y=\sum_{i\leq s}\lambda_i\gamma_i(y)\quad(y\in Y),
 \qquad
 e_{x_i}=\lambda_i\gamma_i(x_i)\quad(i\leq s).
\]
The displayed equality says exactly that $e\in\Rel_G(W)$.

For each $y\in Y$, we have
\[
|e_y|\leq\sum_{i\le s}|\lambda_i|\,|\gamma_i(y)|
\le(s+1)BQ\le M,
\]
and for each $i\leq s$,
\[
|e_{x_i}|=|\lambda_i\gamma_i(x_i)|\le BQ\le M.
\]
Moreover, $e\neq0$: if $\lambda_{i_0}\ne0$, then
$e_{x_{i_0}}=\lambda_{i_0}\gamma_{i_0}(x_{i_0})\ne0$.
Extending $e$ by zero on $X\setminus W$ therefore gives a nonzero element
of $\Rel_G(X)$ of height at most $M$.
This contradicts the $M$-independence of $X$.
Therefore
$|X\setminus Y|\le|A|$, as required.
\end{proof}

\begin{lemma}\label{lem:fusion-step}
Let $G$ be a torsion-free Abelian group, let $A,X\subseteq G$ be finite, and
let $\psi:G\to\T$ be a homomorphism. Fix $\varepsilon>0$ and
$R\in\N$ with $R\geq |A|$.

For each positive integer $m\leq R+|X|$, let $q(m,\varepsilon)$ be an
integer supplied by Lemma~\ref{lem:uniform-kronecker}, and let $Q\in\N$
satisfy
\[
Q\geq
\max\Bigl(
 \{1\}\cup
 \{q(m,\varepsilon):1\leq m\leq R+|X|\}
\Bigr).
\]

Let $M=M(R,Q)$ be supplied by Lemma~\ref{lem:bounded-deletion}. If $X$ is
$M$-independent, then there exist a set $Y\subseteq X$ and a
homomorphism $\psi':G\to\T$ such that
\begin{enumerate}
\item\label{item:fusion-delete}
$|X\setminus Y|\leq |A|$;

\item\label{item:fusion-no-mixed}
there is no mixed $(b,c)\in\Rel_G(A,Y)$ such that $\htop(b,c)\leq Q$;

\item\label{item:fusion-old}
$\|\psi'(a)-\psi(a)\|<\varepsilon$ for every $a\in A$;

\item\label{item:fusion-new}
$\|\psi'(y)\|<\varepsilon$ for every $y\in Y$.
\end{enumerate}
\end{lemma}

\begin{proof}
Suppose that $X$ is $M$-independent. By
Lemma~\ref{lem:bounded-deletion}, there exists a set $Y\subseteq X$
satisfying items~\ref{item:fusion-delete}
and~\ref{item:fusion-no-mixed}.  Fix such a set $Y$.

We first observe that $A\cap Y=\varnothing$. Indeed, if $z\in A\cap Y$, the
coefficient families $b\in\Z^A$ and $c\in\Z^Y$ defined by
\[
 b_z=1,\quad c_z=-1,
 \qquad b_a=0\ (a\in A\setminus\{z\}),
 \qquad c_y=0\ (y\in Y\setminus\{z\})
\]
would form a mixed element $(b,c)\in\Rel_G(A,Y)$ with
$\htop(b,c)=1\leq Q$, contradicting
item~\ref{item:fusion-no-mixed}.

Since $A$ and $Y$ are disjoint, we may define a function
$t:A\cup Y\to\T$ by
\[
t(z)=
\begin{cases}
\psi(z),&\text{if }z\in A,\\
0,&\text{if }z\in Y.
\end{cases}
\]

If $A\cup Y=\varnothing$, take $\psi'=\psi$.  Items~\ref{item:fusion-old}
and~\ref{item:fusion-new} are then vacuously
satisfied.  Suppose now that $A\cup Y\neq\varnothing$, and let
$z=(z_i)_{i<m}$ enumerate this set without repetitions, so that
\[
A\cup Y=\{z_i:i<m\},
\]
for some $m\in\N$. Then, we have
\[
m=|A|+|Y|\leq R+|X|,
\]
and, by the choice of $Q$, $q(m,\varepsilon)\leq Q$.

Let $\mathbf d=(d_i)_{i<m}\in\Rel(z)$ satisfy
$\htop(\mathbf d)\leq q(m,\varepsilon)$. We will prove that
$\sum_{i<m}d_i t(z_i)=0$.  Since $\mathbf d\in\Rel(z)$, we have
\[
\sum_{\substack{i<m\\z_i\in A}}d_i z_i
+
\sum_{\substack{i<m\\z_i\in Y}}d_i z_i
=0,
\]
and, because the enumeration has no repetitions and $A\cap Y=\varnothing$,
it determines $(b,c)\in\Rel_G(A,Y)$ by
\[
 b_{z_i}=d_i\quad(z_i\in A),
 \qquad
 c_{z_i}=d_i\quad(z_i\in Y).
\]
$\htop(b,c)\leq q(m,\varepsilon)\leq Q$.  By
item~\ref{item:fusion-no-mixed}, the pair
$(b,c)$ is not mixed, so one of the following two alternatives holds:
\[
d_i=0\quad\text{for every $i<m$ such that $z_i\in A$},
\]
or
\[
d_i=0\quad\text{for every $i<m$ such that $z_i\in Y$}.
\]

In the first case, all possibly nonzero coefficients correspond to
elements of $Y$. Since $t(y)=0$ for every $y\in Y$, we obtain
\[
\sum_{i<m}d_i t(z_i)=0.
\]

In the second case, all possibly nonzero coefficients correspond to
elements of $A$. Since $\psi$ is a homomorphism, we obtain
\[
\begin{aligned}
\sum_{i<m}d_i t(z_i)
&=
\sum_{\substack{i<m\\z_i\in A}}d_i\psi(z_i)\\
&=
\psi\left(
\sum_{\substack{i<m\\z_i\in A}}d_i z_i
\right)\\
&=\psi(0)=0.
\end{aligned}
\]

Thus, in either case,
\[
\sum_{i<m}d_i t(z_i)=0.
\]

Lemma~\ref{lem:uniform-kronecker} now provides a homomorphism
$\psi':G\to\T$ such that
\[
\|\psi'(z_i)-t(z_i)\|<\varepsilon
\qquad(i<m).
\]
For every $a\in A$, we have $t(a)=\psi(a)$, and hence
\[
\|\psi'(a)-\psi(a)\|<\varepsilon.
\]
For every $y\in Y$, we have $t(y)=0$, and hence $\|\psi'(y)\|<\varepsilon$.
Thus items~\ref{item:fusion-old} and~\ref{item:fusion-new} also hold.
\end{proof}

%% file: sections/05-triangular-coding-and-block-ultrafilters.tex
\section{Block ultrafilters}\label{sec:block-ultrafilters}

Fix a torsion-free Abelian group $G$ of cardinality $\cfrak$ and, using
Lemma~\ref{lem:torsion-free-coordinatization}, identify it so that
\[
 \Z^{(\cfrak)}\leq G\leq\Q^{(\cfrak)}.
\]
There are exactly $\cfrak$ injective sequences in $G$.  Enumerate them without repetitions as
$(s_\alpha)_{\alpha<\cfrak}$.

For each $\alpha<\cfrak$, the union of the supports of the terms of
$s_\alpha$ is countable and hence bounded in $\cfrak$, because
$\operatorname{cf}(\cfrak)>\omega$ by K\"onig's theorem.  Recursively choose
distinct coordinates $\iota(\alpha)<\cfrak$ so that $\iota$ is an injection
and
\[
 \bigcup_{n\in\N}\supp s_\alpha(n)
 \subseteq\iota(\alpha)
 \qquad(\alpha<\cfrak).
\]

We also fix the numerical parameters and the finite blocks used in the
construction as follows.

\begin{itemize}[
    leftmargin=*,
    itemsep=1.25em,
    topsep=0.75em
]

\item 
The sequence $(N_l)_{l\in\N}$ is defined recursively by
\[
N_l
=
(l+2)
\left(
2l+2+\sum_{j<l}N_j
\right)
\qquad(l\in\N),
\]
where the empty sum is understood to be equal to $0$.

\item 
For every $l\in\N$, put
\begin{equation}\label{eq:local-stage-properties2}
S_l=\sum_{j<l}N_j,
\quad
R_l=S_l+2l+2\quad \text{and} \quad \varepsilon_l=2^{-(l+6)}.
\end{equation}

\item 
For every $l\in\N$ and every positive integer
$m\leq R_l+N_l$, choose an integer $q(m,\varepsilon_l)$ supplied by
Lemma~\ref{lem:uniform-kronecker}. Then put
\[
Q_l
=
\max\Bigl(
\{1\}\cup
\{q(m,\varepsilon_l):1\leq m\leq R_l+N_l\}
\Bigr).
\]

\item 
For every $l\in\N$, let
\[
M_l=M(R_l,Q_l)
\]
be the integer supplied by Lemma~\ref{lem:bounded-deletion} for the
parameters $R_l$ and $Q_l$.

\item 
Finally, define
\[
I_l
=
\{S_l,S_l+1,\ldots,S_l+N_l-1\}
\qquad(l\in\N).
\]

\end{itemize}

For every $l\in\N$, the definitions give
\[
 N_l=(l+2)R_l,
 \qquad
 \frac{R_l}{N_l}=\frac1{l+2} \quad\text{and}\quad \sum_{l\in\N}\varepsilon_l=\frac1{32}.
\]
The intervals $(I_l)_{l\in\N}$ are pairwise disjoint and partition
$\N$.

\begin{lemma}\label{lem:independent-block-subsequences}
There exist strictly increasing maps $\varphi_\alpha:\N\to\N$ and sequences
\[
h_\alpha=s_\alpha\circ\varphi_\alpha:
\N\longrightarrow G
\qquad(\alpha<\cfrak)
\]
such that, for every $\alpha<\cfrak$, the following conditions hold:
\begin{itemize}
\item for every $l\in\N$, the set
\[
X_{\alpha,l}
=
\{h_\alpha(n)-e_{\iota(\alpha)}:n\in I_l\}
\]
has cardinality $N_l$ and is $M_l$-independent;

\item for each $n\in\N$,
\begin{equation}\label{eq:prepared-support}
 \supp h_\alpha(n)\subseteq\iota(\alpha).
\end{equation}
\end{itemize}
\end{lemma}

\begin{proof}
Fix $\alpha<\cfrak$, and define $\delta_\alpha:\N\to G$
by
\[
\delta_\alpha(k)
=
s_\alpha(k)-e_{\iota(\alpha)}
\qquad(k\in\N).
\]
The map $\delta_\alpha$ is injective.

\begin{claim}
There exists a sequence $(K_{\alpha,l}:l\in\N)$ of finite subsets of $\N$
satisfying the following conditions:
\begin{enumerate}[label=(\alph*)]
\item $|K_{\alpha,l}|=N_l$;

\item $\max K_{\alpha,j}<\min K_{\alpha,l}$
whenever $j<l$;

\item $\delta_\alpha[K_{\alpha,l}]$ is $M_l$-independent.
\end{enumerate}
\end{claim}
\begin{proof}[Proof of the claim]
Suppose that $K_{\alpha,j}$ has been chosen for every $j<l$. Put
\[
r_{\alpha,l}
=
\begin{cases}
0,&l=0,\\[2mm]
1+\max\displaystyle\bigcup_{j<l}K_{\alpha,j},&l>0.
\end{cases}
\]
Since the set
\[
T_{\alpha,l}
=
\{\delta_\alpha(k):k\geq r_{\alpha,l}\}
\]
is infinite and $G$ is torsion-free, Lemma~\ref{lem:finite-extraction}
provides an $M_l$-independent set
\[
Z_{\alpha,l}\subseteq T_{\alpha,l}
\]
of cardinality $N_l$.

Let
\[
K_{\alpha,l}
=
\delta_\alpha^{-1}[Z_{\alpha,l}].
\]
As $\delta_\alpha$ is injective and
$Z_{\alpha,l}\subseteq T_{\alpha,l}$, we have
$|K_{\alpha,l}|=|Z_{\alpha,l}|=N_l$,
$K_{\alpha,l}\subseteq\N\setminus r_{\alpha,l}$, and
$\delta_\alpha[K_{\alpha,l}]=Z_{\alpha,l}$.
Therefore the three recursive requirements are satisfied.
\end{proof}
For every $l\in\N$, enumerate $K_{\alpha,l}$ increasingly as
\[
K_{\alpha,l}
=
\{k_{\alpha,l,0}
  <k_{\alpha,l,1}
  <\cdots
  <k_{\alpha,l,N_l-1}\}.
\]
We now define an auxiliary map $\varphi_\alpha:\N\to\N$
by
\[
\varphi_\alpha(S_l+j)
=
k_{\alpha,l,j}
\qquad(l\in\N,\ j<N_l).
\]
Since the intervals $I_l$
partition $\N$, this defines $\varphi_\alpha$ on all of $\N$.  The map
$\varphi_\alpha$ is strictly increasing.
Define
\[
h_\alpha=s_\alpha\circ \varphi_\alpha:\N\to G.
\]
For every $l\in\N$, we have
\[
\begin{aligned}
X_{\alpha,l}
&=
\{h_\alpha(n)-e_{\iota(\alpha)}:n\in I_l\}\\
&=
\{s_\alpha(k_{\alpha,l,j})-e_{\iota(\alpha)}:j<N_l\}\\
&=
\delta_\alpha[K_{\alpha,l}]
=
Z_{\alpha,l}.
\end{aligned}
\]
Consequently, $X_{\alpha,l}$ has cardinality $N_l$ and is
$M_l$-independent.

Finally, since $h_\alpha$ is a subsequence of $s_\alpha$,
\[
\supp h_\alpha(n)\subseteq\iota(\alpha)
\qquad(n\in\N).
\]
\end{proof}

Fix sequences $(h_\alpha)_{\alpha<\cfrak}$ as in
Lemma~\ref{lem:independent-block-subsequences} and a pairwise almost disjoint
family
\[
 (A_\alpha)_{\alpha<\cfrak}\subseteq[\N]^\omega;
\]
thus every $A_\alpha$ is infinite and
$|A_\alpha\cap A_\beta|<\infty$ whenever $\alpha\neq\beta$.  It is well known that such a family
exists in ZFC.  For every
$\alpha<\cfrak$, define
\[
\mathscr D_\alpha
=
\left\{
C\subseteq\N:
\lim_{\substack{l\in A_\alpha\\l\to\infty}}
\frac{|I_l\setminus C|}{N_l}=0
\right\}.
\]

For every $\alpha<\cfrak$, the family $\mathscr D_\alpha$ is a free filter on
$\N$.  Indeed, it contains $\N$, does not contain the empty set, and is upward
closed.  If $C,C'\in\mathscr D_\alpha$, then
\[
 |I_l\setminus(C\cap C')|
 \leq |I_l\setminus C|+|I_l\setminus C'|,
\]
so $C\cap C'\in\mathscr D_\alpha$.  Finally, if $F\subseteq\N$ is finite,
then $I_l\cap F=\varnothing$ for all sufficiently large $l$, and hence
$\N\setminus F\in\mathscr D_\alpha$.

\begin{lemma}\label{lem:block-filters}
Let $\alpha<\cfrak$.  Suppose that $B\subseteq A_\alpha$,
$A_\alpha\setminus B$ is finite, and
\[
E_l\subseteq I_l,
\qquad
|E_l|\leq R_l
\qquad(l\in B).
\]
Then $\displaystyle \bigcup_{l\in B}(I_l\setminus E_l)
\in\mathscr D_\alpha$.
\end{lemma}

\begin{proof}
Put $\displaystyle U
=
\bigcup_{l\in B}(I_l\setminus E_l)$.

Since $A_\alpha\setminus B$ is finite, there exists $l_0\in\N$ such
that $A_\alpha\cap[l_0,\infty)\subseteq B$.  Suppose that $l\in A_\alpha$
and $l\geq l_0$.  Then $l\in B$, and, as the intervals
$(I_j)_{j\in\N}$ are pairwise disjoint, we have
$I_l\cap U=I_l\setminus E_l$ and hence $I_l\setminus U=E_l$.
Consequently,
\[
\frac{|I_{l}\setminus U|}{N_{l}}
=
\frac{|E_{l}|}{N_{l}}
\leq
\frac{R_{l}}{N_{l}}
=
\frac{1}{l+2},
\]
and thus
\[
\lim_{\substack{l\in A_\alpha\\l\to\infty}}
\frac{|I_l\setminus U|}{N_l}=0.
\]
By the definition of $\mathscr D_\alpha$, it follows that
$U\in\mathscr D_\alpha$.
\end{proof}

For every $\alpha<\cfrak$, choose an ultrafilter $p_\alpha$ on $\N$
such that $\mathscr D_\alpha\subseteq p_\alpha$.
The ultrafilter $p_\alpha$ is free because $\mathscr D_\alpha$ contains the
cofinite filter.
In particular, whenever $B\subseteq A_\alpha$,
$A_\alpha\setminus B$ is finite, and
\[
E_l\subseteq I_l,
\qquad
|E_l|\leq R_l
\qquad(l\in B),
\]
we have $\bigcup_{l\in B}(I_l\setminus E_l)
\in p_\alpha$.

%% file: sections/06-local-homomorphism-fusion.tex
\section{Local fusion of circle-valued homomorphisms}

Fix $x\in G\setminus\{0\}$.  We shall construct a homomorphism that
does not vanish at $x$ and satisfies the prescribed ultrafilter limits for
all sequences whose distinguished coordinate belongs to a suitable
countable set.

\begin{lemma}\label{lem:countable-dependency-closure}
Let $(D_k)_{k\in\N}$ be a sequence defined recursively as follows:
\[
D_0=\supp x;
\]
\[
D_{k+1}
=
D_k\cup
\bigcup_{\substack{\alpha<\cfrak\\ \iota(\alpha)\in D_k}}
\ \bigcup_{n\in\N}\supp h_\alpha(n).
\]

Let $D=\bigcup_{k\in\N}D_k$.

Then $D$ is countable. Moreover, for every $\alpha<\cfrak$ such that
$\iota(\alpha)\in D$, we have
\[
\supp h_\alpha(n)\subseteq D
\qquad\text{for every }n\in\N.
\]
\end{lemma}

\begin{proof}
A straightforward induction shows that $D_k$ is
countable for every $k\in\N$.  Hence, $D$ is countable.

Now suppose that $\iota(\alpha)\in D$.  Then $\iota(\alpha)\in D_k$ for
some $k$,
and the definition of $D_{k+1}$ gives
$\supp h_\alpha(n)\subseteq D_{k+1}\subseteq D$ for every $n\in\N$.
\end{proof}

From now on, put
\[
 G_D=G\cap\Q^{(D)} \qquad\text{and}\qquad
  \Gamma_D
 =
 \{\alpha<\cfrak:\iota(\alpha)\in D\}.
\]
Choose an enumeration without repetitions
$\Gamma_D=\{\alpha_j:j<J\}$, where $J\leq\omega$. Now, put
\[
 B_{\alpha_j}
 =
 A_{\alpha_j}\setminus\bigcup_{i<j}A_{\alpha_i}
 \qquad(j<J).
\]

It is clear that $(B_{\alpha_j}:j<J)$ are pairwise disjoint and, for each
$j<J$, we have
\begin{equation}\label{eq:local-block-properties}
 B_{\alpha_j}\subseteq A_{\alpha_j}
 \quad\text{and}\quad
 |A_{\alpha_j}\setminus B_{\alpha_j}|<\infty.
\end{equation}

For each $l\in\N$ there is at
most one $\alpha\in\Gamma_D$ such that $l\in B_\alpha$.  We define
\begin{equation}\label{eq:local-stage-set}
 X_l=
 \begin{cases}
  X_{\alpha,l},&\text{if $l\in B_\alpha$ for some $\alpha\in\Gamma_D$},\\
  \varnothing,&\text{if $l\notin\displaystyle\bigcup_{\alpha\in\Gamma_D}B_\alpha$}.
 \end{cases}
\end{equation}

\begin{lemma}
\label{lem:local-stage-properties}
For every $l\in\N$,
\begin{equation}\label{eq:local-stage-properties}
 X_l\subseteq G_D,
 \qquad
 |X_l|\leq N_l,
 \qquad
 X_l\text{ is $M_l$-independent}.
\end{equation}
\end{lemma}

\begin{proof}
If $X_l=\varnothing$, all three assertions are immediate.  Otherwise,
$X_l=X_{\alpha,l}$ for the unique $\alpha\in\Gamma_D$ such that
$l\in B_\alpha$.  Since $\iota(\alpha)\in D$,
Lemma~\ref{lem:countable-dependency-closure}
shows that
\[
 X_{\alpha,l}
 =
 \{h_\alpha(n)-e_{\iota(\alpha)}:n\in I_l\}
 \subseteq G_D.
\]
Then Lemma~\ref{lem:independent-block-subsequences} shows that
$|X_l|=|X_{\alpha,l}|=N_l$ and that $X_l=X_{\alpha,l}$ is
$M_l$-independent.
\end{proof}

Since $D$ is countable, so is $\Q^{(D)}$, and hence $G_D$ is countable.  Fix
a surjective sequence $(g_l)_{l\in\N}$ in $G_D$, repeating elements if
needed.

\begin{lemma}
\label{lem:initial-homomorphism}
There is a homomorphism $\psi_0:G_D\to\T$ such that
\begin{equation}\label{eq:initial-homomorphism}
 \psi_0(x)=\frac12+\Z.
\end{equation}
\end{lemma}

\begin{proof}
Since $G_D$ is torsion-free and $x\neq0$, the cyclic group
$\langle x\rangle$ is isomorphic to $\Z$.  The assignment
\[
 nx\longmapsto \frac n2+\Z
\]
defines a homomorphism from $\langle x\rangle$ to $\T$.  Since $\T$ is an
injective Abelian group, this homomorphism extends to $G_D$.
\end{proof}

\begin{proposition}\label{prop:local-fusion}
There exists a homomorphism $\psi_D:G_D\to\T$ such that
\begin{enumerate}
\item $\psi_D(x)\neq0$;
\item for every $\alpha\in\Gamma_D$,
\begin{equation}\label{eq:local-admissibility}
 p_\alpha\text{-}\lim_n\psi_D(h_\alpha(n))
 =
 \psi_D(e_{\iota(\alpha)}).
\end{equation}
\end{enumerate}
\end{proposition}

\begin{proof}
Let $\psi_0$ be supplied by Lemma~\ref{lem:initial-homomorphism}.  We
recursively construct sequences $(\psi_l)_{l\in\N}$, $(A_l)_{l\in\N}$, and
$(Y_l)_{l\in\N}$ such that, for each $l\in\N$:
\begin{enumerate}
    \item\label{item:local-stage-homomorphism}
    $\psi_l:G_D\to\T$ is a homomorphism;
    \item\label{item:local-stage-control-set}
    $A_l=\bigcup_{j<l}Y_j\cup\{x\}\cup\{g_0,\ldots,g_l\}$;
    \item\label{item:local-stage-subset}
    $Y_l\subseteq X_l$;
    \item\label{item:local-stage-deletion}
    $|X_l\setminus Y_l|\leq R_l$;
    \item\label{item:local-stage-old}
    $\|\psi_{l+1}(a)-\psi_l(a)\|<\varepsilon_l$ for every $a\in A_l$;
    \item\label{item:local-stage-new}
    $\|\psi_{l+1}(y)\|<\varepsilon_l$ for every $y\in Y_l$.
\end{enumerate}
Having defined $(\psi_i)_{i\leq l}$, $(A_i)_{i\leq l}$, and
$(Y_i)_{i<l}$, we show how to define $A_l$, $Y_l$, and $\psi_{l+1}$.

Start by defining $A_l$ as in item~\ref{item:local-stage-control-set}.
At each preceding step
$j<l$, we have $Y_j\subseteq X_j$; hence, by
\eqref{eq:local-stage-properties}, $|Y_j|\leq |X_j|\leq N_j$.
It follows from \eqref{eq:local-stage-properties2} that
\begin{align*}
 |A_l|
 &\leq \sum_{j<l}|Y_j|+1+(l+1)\\
 &\leq \sum_{j<l}N_j+l+2\\
 &=S_l+l+2\leq R_l.
\end{align*}

We now apply Lemma~\ref{lem:fusion-step} with
$G=G_D$, $A=A_l$, $X=X_l$, $\psi=\psi_l$,
$\varepsilon=\varepsilon_l$, $R=R_l$, and $Q=Q_l$.  In the notation of
that lemma, $M=M(R_l,Q_l)=M_l$.  This application is valid because
$|A_l|\leq R_l$, $X_l$ is $M_l$-independent, and, since
$|X_l|\leq N_l$, every positive integer $m\leq R_l+|X_l|$ satisfies
$m\leq R_l+N_l$, so $q(m,\varepsilon_l)\leq Q_l$ by the definition of
$Q_l$.  We obtain a set
$Y_l\subseteq X_l$ and a homomorphism
$\psi_{l+1}:G_D\to\T$ satisfying
items~\ref{item:local-stage-subset}--\ref{item:local-stage-new}.  This
completes the recursive construction.

We next prove that the sequence $(\psi_l(g))_{l\in\N}$ converges for
every $g\in G_D$.  Choose $m$ such that $g_m=g$.  For every $l\geq m$
we have $g\in A_l$, and therefore
\[
 \|\psi_{l+1}(g)-\psi_l(g)\|<\varepsilon_l.
\]
Since $\sum_l\varepsilon_l<\infty$, the sequence
$(\psi_l(g))_l$ is Cauchy in $\T$.  The circle group is a complete metric
space, so a limit exists.  Define
\[
 \psi_D(g)=\lim_{l\to\infty}\psi_l(g)
 \qquad(g\in G_D).
\]

As addition in $\T$ is continuous and $\psi_D$ is the pointwise limit of
homomorphisms, $\psi_D$ is a homomorphism.

Since $x\in A_l$ for every $l$,
\eqref{eq:initial-homomorphism} and
item~\ref{item:local-stage-old} give
\[
 \left\|\psi_D(x)-\left(\frac12+\Z\right)\right\|
 \leq
 \sum_{l\in\N}\|\psi_{l+1}(x)-\psi_l(x)\|
 \leq
 \sum_{l\in\N}\varepsilon_l
 =\frac1{32}<\frac12.
\]
In particular, $\psi_D(x)\neq0$.

\begin{claim}
For all $l\in\N$ and all $y\in Y_l$, we have
\begin{equation}\label{eq:small-value-bound}
 \|\psi_D(y)\|<2\varepsilon_l.
\end{equation}
\end{claim}
\begin{proof}[Proof of the claim]
Fix $l\in\N$ and $y\in Y_l$.  For every $k>l$,
item~\ref{item:local-stage-control-set} gives $y\in A_k$, so
item~\ref{item:local-stage-old} gives
$\|\psi_{k+1}(y)-\psi_k(y)\|<\varepsilon_k$.  Therefore,
\begin{align*}
 \|\psi_D(y)\|
 &\leq
 \|\psi_{l+1}(y)\|+\sum_{k>l}\|\psi_{k+1}(y)-\psi_k(y)\|\\
 &<\varepsilon_l+\sum_{k>l}\varepsilon_k=2\varepsilon_l.
\end{align*}
\end{proof}

Now we verify \eqref{eq:local-admissibility}.  Fix
$\alpha\in\Gamma_D$.  If $l\in B_\alpha$, then the definition
\eqref{eq:local-stage-set} gives $X_l=X_{\alpha,l}$.  Define
\[
 E_l
 =
 \{n\in I_l:h_\alpha(n)-e_{\iota(\alpha)}\notin Y_l\}
 \qquad(l\in B_\alpha).
\]
The map $n\longmapsto h_\alpha(n)-e_{\iota(\alpha)}$
is a bijection from $I_l$ onto $X_{\alpha,l}=X_l$.  Consequently, by
item~\ref{item:local-stage-deletion},
\[
 |E_l|
 =
 |X_l\setminus Y_l|
 \leq R_l
 \qquad(l\in B_\alpha).
\]
Equation~\eqref{eq:local-block-properties} and
Lemma~\ref{lem:block-filters} now give
$\displaystyle
 U_\alpha
 =
 \bigcup_{l\in B_\alpha}(I_l\setminus E_l)
 \in p_\alpha.
$

We now intend to show that
\begin{equation}\label{eq:local-limit}
 p_\alpha\text{-}\lim_n
 \psi_D(h_\alpha(n)-e_{\iota(\alpha)})
 =0.
\end{equation}

Let $n\in U_\alpha$, and let $l$ be the unique index such that
$n\in I_l$.  Then $l\in B_\alpha$, $n\notin E_l$, and hence
$h_\alpha(n)-e_{\iota(\alpha)}\in Y_l$.
It follows from \eqref{eq:small-value-bound} that
\begin{equation}\label{eq:block-limit-estimate}
 \|\psi_D(h_\alpha(n)-e_{\iota(\alpha)})\|
 \leq 2\varepsilon_l.
\end{equation}

Fix $\delta>0$ and choose $L\in\N$ such that $2\varepsilon_L<\delta$.
The set $F_L=\bigcup_{l<L}I_l$ is finite, so, since $p_\alpha$ is free,
\[
 V_L=U_\alpha\cap(\N\setminus F_L)\in p_\alpha.
\]
If $n\in V_L$ and $n\in I_l$, then $n\notin F_L$ implies $l\geq L$.
Since $(\varepsilon_l)_l$ is decreasing, \eqref{eq:block-limit-estimate}
gives
\[
 \|\psi_D(h_\alpha(n)-e_{\iota(\alpha)})\|
 \leq 2\varepsilon_l\leq 2\varepsilon_L<\delta.
\]
Since $\delta>0$ was arbitrary, \eqref{eq:local-limit} follows.

Finally, the additivity of $\psi_D$ gives
\[
 \psi_D(h_\alpha(n))
 =
 \psi_D(h_\alpha(n)-e_{\iota(\alpha)})
 +
 \psi_D(e_{\iota(\alpha)}),
\]
and taking the $p_\alpha$-limit proves
\eqref{eq:local-admissibility}.
\end{proof}

The argument uses only that the local group $G_D$ is countable,
torsion-free, and Abelian.

%% file: sections/07-transfinite-extension.tex
\section{Extending the homomorphisms}

The previous section solves the problem on a countable subgroup: for each
nonzero $x\in G$, it constructs a homomorphism on $G_D$ that does not vanish
at $x$ and satisfies every prescribed limit equation associated with a
coordinate in $D$, where $D$ is a suitable countable subset of $\cfrak$.
We now extend this homomorphism to all of $G$.

\begin{proposition}\label{prop:transfinite-extension}
Let $D\subseteq\cfrak$ be such that, for every $\alpha<\cfrak$,
\[
 \iota(\alpha)\in D\quad\Longrightarrow\quad
 \supp h_\alpha(n)\subseteq D\quad\text{for every }n\in\N.
\]
Suppose that $\psi_D:G_D\to\T$ satisfies
\eqref{eq:local-admissibility} whenever $\iota(\alpha)\in D$.  Then
$\psi_D$ extends to a homomorphism $\psi:G\to\T$ satisfying
\[
 p_\alpha\text{-}\lim_n\psi(h_\alpha(n))
 =\psi(e_{\iota(\alpha)})
 \qquad(\alpha<\cfrak).
\]
\end{proposition}

\begin{proof}
Since $\T$ is an injective Abelian group, $\psi_D$ extends from the subgroup
$G_D$ of $\Q^{(D)}$ to a homomorphism
\[
 \widetilde\psi_D:\Q^{(D)}\longrightarrow\T.
\]

By transfinite recursion, we construct homomorphisms
$\theta_\xi:\Q\to\T$ for each $\xi<\cfrak$.
If $\xi\in D$, let
\[
 \theta_\xi(q)=\widetilde\psi_D(qe_\xi)
 \qquad(q\in\Q).
\]
Suppose next that $\xi\notin D$ and $\xi\in\iota[\cfrak]$.
Since $\iota$ is injective, there is a unique $\alpha<\cfrak$ such that
$\xi=\iota(\alpha)$.  By \eqref{eq:prepared-support}, all coordinates in
$h_\alpha(n)$ precede $\xi$, so the values
\begin{equation}\label{eq:stage-evaluation}
 w_{\alpha,n}=
 \sum_{\eta\in\supp h_\alpha(n)}
 \theta_\eta(h_\alpha(n)(\eta))
\end{equation}
are already defined.  Let
\begin{equation}\label{eq:compact-stage-limit}
 t_\alpha=p_\alpha\text{-}\lim_n w_{\alpha,n},
\end{equation}
which exists by compactness of $\T$.

Define $\rho:\Z\to\T$ by $\rho(k)=kt_\alpha$.  The injectivity of $\T$
supplies an extension $\theta_\xi:\Q\to\T$.  In particular,
\begin{equation}\label{eq:coordinate-one}
 \theta_\xi(1)=t_\alpha.
\end{equation}
If $\xi\notin D\cup\iota[\cfrak]$, put $\theta_\xi=0$.

The family $(\theta_\xi)_{\xi<\cfrak}$ induces a homomorphism on the
direct sum $\Q^{(\cfrak)}$, namely
\begin{equation}\label{eq:global-homomorphism}
 \widetilde\psi:\Q^{(\cfrak)}\longrightarrow\T,
 \qquad
 \widetilde\psi(z)=
 \sum_{\xi\in\supp z}\theta_\xi(z(\xi)).
\end{equation}
Finally, let $\psi=\widetilde\psi\mathbin{\upharpoonright}G$.

We first verify that $\psi$ extends $\psi_D$.  If $z\in\Q^{(D)}$, then
\[
 \widetilde\psi(z)
 =\sum_{\xi\in\supp z}\theta_\xi(z(\xi))
 =\sum_{\xi\in\supp z}\widetilde\psi_D(z(\xi)e_\xi)
 =\widetilde\psi_D(z).
\]
Thus $\widetilde\psi$ extends $\widetilde\psi_D$, so its restriction $\psi$
extends $\psi_D$.

We now verify the ultrafilter limit identities.  Fix $\alpha<\cfrak$.
Suppose first that $\iota(\alpha)\in D$.  The hypothesis on $D$ gives
$\supp h_\alpha(n)\subseteq D$ for every $n\in\N$.  Hence
$h_\alpha(n),e_{\iota(\alpha)}\in G_D$, and, since $\psi$ extends $\psi_D$,
\eqref{eq:local-admissibility} gives
\[
 p_\alpha\text{-}\lim_n\psi(h_\alpha(n))
 =\psi(e_{\iota(\alpha)}).
\]

Suppose now that $\iota(\alpha)\notin D$.  By
\eqref{eq:prepared-support},
\[
 \supp h_\alpha(n)\subseteq\iota(\alpha)
 \qquad(n\in\N),
\]
so every coordinate occurring in $h_\alpha(n)$ precedes
$\iota(\alpha)$.  Consequently, \eqref{eq:global-homomorphism} and
\eqref{eq:stage-evaluation} give
\[
 \psi(h_\alpha(n))
 =\widetilde\psi(h_\alpha(n))
 =w_{\alpha,n}.
\]
Therefore, by \eqref{eq:compact-stage-limit} and
\eqref{eq:coordinate-one},
\[
 p_\alpha\text{-}\lim_n\psi(h_\alpha(n))
 =t_\alpha
 =\theta_{\iota(\alpha)}(1)
 =\psi(e_{\iota(\alpha)}).
\]
\end{proof}

Proposition~\ref{prop:local-fusion} provides a homomorphism
$\psi_D:G_D\to\T$ that does not vanish at the chosen nonzero vector $x$ and
satisfies \eqref{eq:local-admissibility}.  By
Proposition~\ref{prop:transfinite-extension}, it extends to a homomorphism on
$G$ satisfying all the limit identities stated there.

\begin{corollary}\label{cor:admissible-homomorphisms}
For every $0\neq x\in G$ there exists a homomorphism
$\psi_x:G\to\T$ such that
\begin{enumerate}
\item\label{item:detect-x} $\psi_x(x)\neq0$;
\item\label{item:all-admissible} for every $\alpha<\cfrak$,
\[
 p_\alpha\text{-}\lim_n\psi_x(h_\alpha(n))
 =\psi_x(e_{\iota(\alpha)}).
\]
\end{enumerate}
\end{corollary}

%% file: sections/08-group-topologies.tex
\section{The group topologies}\label{sec:group-topologies}

For each $0\neq x\in G$ choose a homomorphism $\psi_x$ as in
Corollary~\ref{cor:admissible-homomorphisms}.  Define
\[
 \Delta_G:G\longrightarrow
 \T^{\,G\setminus\{0\}},
 \qquad
 \Delta_G(z)=(\psi_x(z))_{x\neq0},
\]
and give $G$ the initial topology $\tau_G$ induced by this map.

\begin{lemma}\label{lem:initial-topology}
The topology $\tau_G$ is a Tychonoff group topology.  For
every $\alpha<\cfrak$, the free ultrafilter $p_\alpha$ satisfies
\begin{equation}\label{eq:group-ultrafilter-limit}
 p_\alpha\text{-}\lim_n h_\alpha(n)=e_{\iota(\alpha)}.
\end{equation}
\end{lemma}

\begin{proof}
The initial topology of a homomorphism into a topological group is a group
topology.  The diagonal is injective: if $z\neq w$, then $x=z-w\neq0$, and
Corollary~\ref{cor:admissible-homomorphisms}\ref{item:detect-x} gives
$\psi_x(x)\neq0$, so
$\psi_x(z)\neq\psi_x(w)$.  Since $\T$ is Hausdorff, the induced
topology is Hausdorff.  In fact, $\Delta_G$ identifies $(G,\tau_G)$ with a
subspace of a power of the compact Hausdorff group $\T$, so
$(G,\tau_G)$ is Tychonoff.

For a fixed $\alpha<\cfrak$,
Corollary~\ref{cor:admissible-homomorphisms}\ref{item:all-admissible} gives
\[
 p_\alpha\text{-}\lim_n\psi_x(h_\alpha(n))
 =\psi_x(e_{\iota(\alpha)})\qquad(x\neq0).
\]
Hence $\Delta_G(h_\alpha(n))$ converges along $p_\alpha$ to
$\Delta_G(e_{\iota(\alpha)})$, since convergence in a product is
coordinatewise.  As $\tau_G$ is induced by the injective map
$\Delta_G$, this is precisely \eqref{eq:group-ultrafilter-limit}.
\end{proof}

We are now ready to give the proofs of Theorem~\ref{thm:main} and
Proposition~\ref{prop:rational}.

\begin{proof}[Proof of the main results]
Lemma~\ref{lem:initial-topology} shows that $\tau_G$ is a Hausdorff group
topology.  Let $s:\N\to G$ be injective.  The enumeration in
Section~\ref{sec:block-ultrafilters} gives a unique $\alpha<\cfrak$ with
$s=s_\alpha$.  By
Lemma~\ref{lem:independent-block-subsequences}, the map
$\varphi_\alpha:\N\to\N$ is strictly increasing and
$h_\alpha=s\circ\varphi_\alpha$.  By
Lemma~\ref{lem:initial-topology} and
\eqref{eq:group-ultrafilter-limit}, the free ultrafilter $p_\alpha$
satisfies
\[
 p_\alpha\text{-}\lim_n s\bigl(\varphi_\alpha(n)\bigr)
 =e_{\iota(\alpha)}\neq0.
\]
Thus all the hypotheses of Lemma~\ref{lem:topological-reduction} hold,
so $(G,\tau_G)$ is countably compact and every convergent sequence in it is
eventually constant.  This proves Theorem~\ref{thm:main}.  Taking
$G=\Q^{(\cfrak)}$ and transporting the resulting topology along an algebraic
isomorphism $\Q^{(\cfrak)}\cong(\mathbb R,+)$ gives
Proposition~\ref{prop:rational}.
\end{proof}

%% file: sections/09-wallace-semigroup.tex
\section{The Wallace semigroup}

Specialize the preceding construction to $G=\Z^{(\cfrak)}$, with its
canonical basis, and let $\tau=\tau_G$ be the resulting topology.  Consider
\[
 P=\{z\in\Z^{(\cfrak)}:z(\xi)\geq0\text{ for every }\xi<\cfrak\}\subseteq \Z^{(\cfrak)}.
\]
Of course, $P$ is a subsemigroup of $\mathbb Z^{(\cfrak)}$ under addition, so
it is two-sided cancellative and Tychonoff, and it is clearly not a group.
It remains to see that $P$ is countably compact.

\begin{proof}[Proof of Corollary~\ref{cor:wallace}]
Let $A\subseteq P$ be countably infinite and choose an injective enumeration
$s:\N\to A$.  Let $\alpha<\cfrak$ be the unique index for which
$s=s_\alpha$.  The subsequence
$h_\alpha=s_\alpha\circ\varphi_\alpha$ remains in $A$, and
\[
 p_\alpha\text{-}\lim_n h_\alpha(n)=e_{\iota(\alpha)}
\]
in $\Z^{(\cfrak)}$.  Both the sequence and its limit lie in $P$, so the same
convergence holds in the subspace topology.  Since $p_\alpha$ is free and
$h_\alpha$ is injective, $e_{\iota(\alpha)}$ is an accumulation point of
$A$ in $P$.
Thus, $P$ is countably compact.
\end{proof}

%% file: sections/10-consequences-and-concluding-remarks.tex
\section{Further consequences and concluding remarks}\label{sec:consequences}

We collect consequences that use the main theorem together with earlier
reductions.  We begin with an intrinsic strengthening of the topology
constructed in Section~\ref{sec:group-topologies}.

\begin{proposition}\label{prop:large-closures}
Let $G$ be a torsion-free Abelian group of cardinality $\cfrak$, endowed with
the topology $\tau_G$ constructed in
Section~\ref{sec:group-topologies}.  If $A\subseteq G$ is
infinite, then
\[
 |\overline A|=\cfrak.
\]
Consequently, every infinite closed subset of $(G,\tau_G)$ has cardinality
$\cfrak$.
\end{proposition}

\begin{proof}
Choose a countably infinite set $B\subseteq A$.  There are $\cfrak$ bijections
from $\N$ onto $B$.  Each of them occurs in the enumeration
$(s_\alpha)_{\alpha<\cfrak}$ fixed in Section~\ref{sec:block-ultrafilters}.
For such an index $\alpha$, the sequence
$h_\alpha=s_\alpha\circ\varphi_\alpha$ takes its values in $B$, and
Lemma~\ref{lem:initial-topology} gives
\[
 p_\alpha\text{-}\lim_n h_\alpha(n)=e_{\iota(\alpha)}.
\]
Since $p_\alpha$ is free, $e_{\iota(\alpha)}\in\overline B$.  The map $\iota$
is injective, so the $\cfrak$ bijections from $\N$ onto $B$ yield $\cfrak$
distinct points of $\overline B$.  Hence
\[
 \cfrak\leq |\overline B|
 \leq |\overline A|
 \leq |G|=\cfrak.
\]
\end{proof}

Shakhmatov asked whether there exists in ZFC a countably compact free Abelian
group in which every infinite closed subset has cardinality at least $\cfrak$
\cite[p.~13]{Shakhmatov2001}.  For $G=\Z^{(\cfrak)}$,
Proposition~\ref{prop:large-closures} shows that the topology constructed here
provides such a group, thereby answering the question in ZFC\@.

The bicyclic semigroup $C(p,q)$ is the monoid generated by two elements $p$
and $q$ subject to the single relation $qp=1$; every element of $C(p,q)$ has
a unique representation $p^nq^m$ with $n,m\in\N$.  Banakh, Dimitrova, and
Gutik proved that the existence of a torsion-free Abelian countably compact
topological group without nontrivial convergent sequences implies the
existence of a Tychonoff countably compact topological semigroup containing a
copy of the bicyclic semigroup
\cite[Theorem~6.6]{BanakhDimitrovaGutik2010}.  Their Problem~7.1 asks whether
such a semigroup exists in ZFC.

\begin{corollary}\label{cor:bicyclic}
In ZFC there exists a Tychonoff countably compact topological semigroup
containing a copy of the bicyclic semigroup.
\end{corollary}

\begin{proof}
Theorem~\ref{thm:main}, applied, for example, to $\Z^{(\cfrak)}$, supplies the
group required in \cite[Theorem~6.6]{BanakhDimitrovaGutik2010}.
\end{proof}

Thus Problem~7.1 of \cite{BanakhDimitrovaGutik2010} is answered in ZFC\@.

There is a parallel consequence for paratopological groups.  Recall that a
paratopological group is a group endowed with a topology for which
multiplication is continuous.  Guran
asked whether there exists a Hausdorff countably compact paratopological group
which is not a topological group and whether every Hausdorff countably compact
paratopological group is totally bounded
\cite[Problems~1 and~2]{Guran1998}.  Banakh and Ravsky isolated the hypothesis
\(\mathrm{TT}\): the existence of an infinite torsion-free Abelian countably
compact topological group without nontrivial convergent sequences.  Under
\(\mathrm{TT}\), their Lemma~3.21 and Example~3.22 construct a functionally
Hausdorff countably compact paratopological group which is not a topological
group and whose underlying abstract group is free Abelian
\cite{BanakhRavsky2020}.

\begin{corollary}\label{cor:paratopological}
In ZFC there exists a functionally Hausdorff countably compact
paratopological group which is not a topological group and whose underlying
abstract group is free Abelian.
\end{corollary}

\begin{proof}
Theorem~\ref{thm:main} establishes \(\mathrm{TT}\) in ZFC\@.  Apply
\cite[Lemma~3.21 and Example~3.22]{BanakhRavsky2020}.
\end{proof}

This answers both problems of Guran.  Indeed, Banakh and Ravsky show as part
of the construction that the example is not left-$\omega$-precompact; in
particular, it is not totally bounded
\cite[Example~3.22 and the discussion following it]{BanakhRavsky2020}.  It
also answers all three clauses of
\cite[Problem~3.23]{BanakhRavsky2020}: the example is Hausdorff and countably
compact; it is Hausdorff, countably pracompact, and Baire; and it is $T_1$ and
countably compact.  In the terminology of
\cite[Problems~2.4.1 and~2.4.3]{ArhangelskiiTkachenko2008}, it gives a ZFC
Hausdorff countably compact paratopological group which is not a topological
group and is not precompact.

Another consequence is a monothetic refinement of the Wallace phenomenon.
A topological monoid $M$ is monothetic if it is topologically generated by one
element, that is, if $M=\overline{\{a^n:n\in\N\}}$ for some $a\in M$.  Banakh,
Bardyla, Guran, Gutik, and Ravsky asked whether the local compactness assumption
in their theorem on monothetic monoids can be weakened to local countable
compactness.  Under \(\mathrm{TT}\), they constructed a Hausdorff
countably compact monothetic topological monoid which embeds into a topological
group but is not a group \cite[Remark~4.5]{BanakhBardylaGuranGutikRavsky2020}.

\begin{corollary}\label{cor:monothetic-monoid}
In ZFC there exists a Hausdorff countably compact monothetic topological
monoid which is a submonoid of a topological group but is not a group.
\end{corollary}

\begin{proof}
Again Theorem~\ref{thm:main} establishes \(\mathrm{TT}\), so the construction
in \cite[Remark~4.5]{BanakhBardylaGuranGutikRavsky2020} applies.
\end{proof}

Finally, we record a consequence for \emph{suitable sets}.  A suitable set in a
topological group is a discrete subset whose union with the identity is closed
and which topologically generates the group.  Dikranjan asked, in particular,
whether there is a topologically finitely generated topological group without
infinite suitable sets \cite[Question~5.4]{Dikranjan1999}.

\begin{proposition}\label{prop:suitable-sets}
In ZFC there exists an infinite monothetic countably compact topological group
without nontrivial convergent sequences and without infinite suitable sets.
\end{proposition}

\begin{proof}
Let $(G,\tau_G)$ be supplied by Theorem~\ref{thm:main}, choose $0\neq g\in G$,
and put $H=\overline{\langle g\rangle}$.  The group $H$ is infinite and
monothetic: the cyclic group $\langle g\rangle$ is infinite because $G$ is
torsion-free.  The subgroup $H$ is closed in $G$, hence countably compact, and
as a subspace of $G$ it has no nontrivial convergent sequences.

Suppose that $S\subseteq H$ were an infinite suitable set and choose distinct
points $s_n\in S$.  We claim that $s_n\to0$.  Let $U$ be an open neighborhood
of $0$.  If infinitely many $s_n$ lay outside $U$, their set would have an
accumulation point in the closed countably compact subspace $H\setminus U$.
Since $S\cup\{0\}$ is closed, that accumulation point would belong to
$S\cup\{0\}$; it could not be $0$, and it could not belong to $S$ because
$S$ is discrete.  This is a contradiction.  Hence all but finitely many
$s_n$ belong to $U$, proving the claim.  The resulting nontrivial convergent
sequence contradicts the defining property of $G$.  Therefore $H$ has no
infinite suitable set.
\end{proof}

Proposition~\ref{prop:suitable-sets} answers the example clause of
\cite[Question~5.4]{Dikranjan1999} in ZFC.

%% file: sections/10-acknowledgments.tex
\section*{Acknowledgments}

The second author was supported by the S\~ao Paulo Research Foundation
(FAPESP), grant no.~2025/07302-0.

\section*{Declaration of generative AI and AI-assisted technologies in the
manuscript preparation process}

OpenAI Codex, powered by the GPT-5.6 Sol model, was used extensively as a
generative research and writing tool.  In an iterative process directed by the
authors, the model carried out most of the exploratory proof development and
generated most of the first draft of the manuscript and the accompanying Lean~4
formalization.  The authors formulated the problem and objectives, provided
the mathematical context and constraints, steered the iterations through
prompts and feedback, supplied ideas and suggested strategies, selected and
revised the outputs, reviewed the final arguments, and extensively rewrote the
text, using AI again for proofreading and review.  The authors take full
responsibility for the accuracy,
originality, and integrity of the work.

\section*{Dedication}
The authors dedicate this paper to Professor Artur Hideyuki Tomita (University
of S\~ao Paulo, Brazil), who supervised them during their doctoral studies.
The authors' earlier research under his guidance provided invaluable preparation for the present work.
Professor Tomita made significant contributions to the problems studied in this paper, and many of the arguments presented here are inspired by his previous work.
Without Professor Tomita's work, the present paper would not have been
possible.